\documentclass[12pt]{article}%
\usepackage{amsfonts}
\usepackage{amsmath}
\usepackage{amssymb}
\usepackage{mathptmx}
\usepackage{graphicx}
\usepackage{url}
\usepackage{subfigure}
\usepackage{color}%
\providecommand{\U}[1]{\protect\rule{.1in}{.1in}}
\newtheorem{theorem}{Theorem}

\newtheorem{corollary}[theorem]{Corollary}

\newtheorem{definition}[theorem]{Definition}

\newtheorem{lemma}[theorem]{Lemma}

\newtheorem{proposition}[theorem]{Proposition}
\newtheorem{remark}[theorem]{Remark}

\newenvironment{proof}[1][Proof]{\textbf{#1.} }{\ \rule{0.5em}{0.5em}}

\begin{document}

\title{Putting all eggs in one basket: some insights from a correlation inequality}
\author{Pradeep Dubey$^{1,4}$, Siddhartha Sahi$^{2,1}$, and Guanyang Wang$^{3}$\\
\\1. Center for Game Theory, Department of Economics \\Stony Brook University, Stony Brook, NY 11790\\2. Department of Mathematics, Rutgers University, \\3. Department of Statistics, Rutgers University, \\New Brunswick, NJ 08903\\4. Cowles Foundation, Yale University, New Haven, CT 06520}
\date{}
\maketitle

\section*{ Abstract}

We give examples of situations{ (stochastic production, military
tactics, corporate merger)} where it is beneficial to concentrate risk rather
than to diversify it, \textit{i.e.}, to put all eggs in one basket. Our
examples admit a dual interpretation: as optimal strategies of a single player
(the \textquotedblleft principal\textquotedblright) or, alternatively, as
dominant strategies in a non-cooperative game with multiple players (the
\textquotedblleft agents\textquotedblright).

The key mathematical result can be formulated in terms of a convolution
structure on the set of increasing functions on a Boolean lattice (the lattice
of subsets of a finite set). This generalizes the well-known Harris inequality
from statistical physics and discrete mathematics; we give a simple
self-contained proof of this result, and prove a \emph{further} generalization
based on the game-theoretic approach.

\bigskip

\textbf{Keywords: }Harris inequality, correlation inequality, increasing
functions, diversifying risk, concentrating risk, stochastic production,
military tactics, corporate merger, dominant strategy. \medskip

\textbf{Mathematics Classification -- MSC2020}: 60E15, 91A10, 91A18, 91B43.

\textbf{Economics Classification -- JEL}: C60, C61, C72, G10.

\section*{Introduction}

\begin{quotation}
{\small \textquotedblleft It is the part of a wise man to keep himself today
for tomorrow, and not to venture all his eggs in one basket." ---} {\small M.
Cervantes (Don Quixote, 1605)}

{\small \textquotedblleft Behold, the fool saith, \textquotedblleft Put not
all thine eggs in the one basket\textquotedblright\ --- which is but a matter
of saying, \textquotedblleft Scatter your money and your
attention\textquotedblright; but the wise man saith, \textquotedblleft Put all
your eggs in the one basket and -- WATCH THAT BASKET.\textquotedblright\ ---
\ M. Twain (Pudd'nhead Wilson, 1894) }
\end{quotation}

Although modern economic theory and practice widely advocate for risk
diversification --- epitomized by the maxim of distributing 'eggs' across
various 'baskets' and embraced alike by academic scholars and by financial
 {practitioners} such as bankers, investors, and portfolio managers ---
there are  {scenarios} where it may be beneficial to concentrate risk.
This is especially true when the reward
from the joint success of all ventures can eclipse the combined gains
from several partial successes. Our analysis aims to
delineate conditions under which it is best to concentrate all resources into
one endeavor, {\it i.e.} to put all eggs in one basket.

In the realm of portfolio theory, the typical payoff function is the
\emph{sum} of returns from various investments, and then it is beneficial to
create strategic diversification by bundling together asset classes that
exhibit no correlation or negative correlation. In contrast, our study
explores scenarios where the payoff is the \emph{product} of exogenously given
functions, and is influenced by strategic choices that affect their joint
distribution. The pivotal insight is that positive correlation among these
functions could make it strategically beneficial to adopt a coordinated
approach, which effectively concentrates risk and increases the probability
that the peaks and troughs of the individual functions occur together.

We present a series of examples demonstrating that such scenarios can emerge
quite naturally in a strategic setting. Our examples span a variety of
scenarios, including stochastic production with unpredictable input
supplies, military tactics aimed at disrupting enemy communication networks,
and the decision-making processes in corporate mergers. Each scenario can be
interpreted in two ways: either as an optimal strategy for an individual actor
(the 'principal') or as dominant strategies in a non-cooperative setting
involving several participants (the 'agents').

In every scenario presented, the principal's potential strategies are
subsets $S$ of a finite set $H$. The case where $S=\emptyset$ represents the
strategy of maximum risk diversification, while $S=H$ signifies the strategy
of maximum risk concentration. We demonstrate that the payoff
$\Pi\left(  S\right)  $ is an increasing function, meaning that if $S$
contains $T$ then $\Pi(S)\geq\Pi(T)$. Consequently, the optimal strategy is to
choose $S=H$, effectively \textquotedblleft putting all eggs in one
basket\textquotedblright.

While the initial examples are binary in nature, Section \ref{sec:many}
expands the discussion to a more complex scenario of stochastic production
with multiple inputs. Here the principal's strategies correspond to
partitions of the input set. We establish in Proposition
\ref{thm: dominant} that the optimal strategy involves opting for the coarsest
partition. To this end, in Section \ref{GTperspective} we recast the
problem as a strategic game among the agents. Proposition \ref{thm: dominant}
 in fact asserts that selecting the coarsest partition is a dominant
strategy for each agent, not just in the standard ex ante sense, but in a much
stronger ex post sense---see Remark \ref{ex post optimality} for the precise statement.

Proposition \ref{thm: dominant} has a further implication. In the
scenario of stochastic production, as detailed in Section
\ref{sec:many}, it may well happen that the principal does not have the
possibility of directly executing her optimal
strategy. Instead, she must rely upon each of her autonomous agents to
\textquotedblleft fall in line\textquotedblright, i.e., to voluntarily choose
to implement his component of her optimal strategy. Remark
\ref{incentive compatibility} illustrates that by allocating a modest share of
her payoff to each agent, the principal can incentivize them to do
precisely this, so that her optimal strategy becomes effectively
\textquotedblleft self-enforcing\textquotedblright.

The unifying mathematical principle for the first three examples is as
follows: Let $\mathcal{H}$ be the power set of a finite set $H$, and let $A$
be the space of real-valued functions on $\mathcal{H}$. In (\ref{=ES}) below,
we define a convolution operation $f \star g$ on $A$, predicated on a 'coin
tossing' mechanism with probabilities $0 \leq p_{h} \leq1$ for each element $h
\in H$. Theorem \ref{th:main} establishes that the convolution of two
increasing functions results in another increasing function. Now it turns out
that in each example, the payoff function $\Pi$ can be expressed as the
convolution of two increasing functions, confirming that $\Pi$ itself is an
increasing function.

We point out that Theorem \ref{th:main} belongs to a class of correlation
inequalities that are widely studied and applied in combinatorics, graph
theory, and statistical physics. In particular it readily implies the
well-known Harris inequality \cite{Harris}, which is a pivotal concept in
percolation theory and the Erdos-Renyi model of random graphs
\cite{Alon-Spencer}. We hope that our paper will serve to introduce the
beautiful subject of correlation inequalities to those previously unacquainted
with the topic.

The Harris inequality and its extensions, such as the
Fortuin-Kasteleyn-Ginibre inequality \cite{Fortuin} and the Ahlswede-Daykin
four-function inequaliy \cite{AD}, have been applied in economic theory to
several areas including comparative statics \cite{Athey, Quah}, bargaining
networks \cite{Azar}, and optimal assignments \cite{Nikzad}. However their
role in risk concentration strategies, which is the central theme of our
paper, remains unexplored. While the paper \cite{Postlewaite} does discuss
risk concentration, it is underpinned by a different set of principles, and
does not employ correlation inequalities.

\section*{Acknowledgement}

The authors thank Ioannis Karatzas, Larry Samuelson, and Eran Shmaya for their
helpful and insightful comments on an earlier version of this paper. The
research of S. Sahi was partially supported by NSF grant DMS-2001537.

\section{Example 1: Stochastic production with two inputs}

An entrepreneur has a Cobb-Douglas (log-linear) production function of the
form
\begin{equation}
\label{=CD}f(x,y)=x^{\alpha}y^{\beta},\quad\alpha,\beta>0,
\end{equation}
involving two inputs $x, y$, which she sources from a finite set $H$ of suppliers.

Each supplier $h\in H$ has $x_{h}$ units of the $x$-input and $y_{h}$ units of
the $y$-input which he needs to ship to the entrepreneur, either separately by
way of two independent shipments, or together in one \textquotedblleft
pooled\textquotedblright\ shipment.
The shipments have zero dollar cost (see, however, Remark \ref{costs}) but are
fraught with risk: the probability that a shipment by $h$ will reach its
destination is $p_{h}.$ These probabilities are independent across the
suppliers and also across different shipments by the same supplier.

The entrepreneur can make a \emph{separating contract} with supplier $h$ for
him to ship $x_{h}$ and $y_{h}$ independently; or a \emph{pooling}
\emph{contract }for him to ship them jointly. She has to give all her
suppliers sufficient advance notice and make these contracts \textit{ex ante},
prior to the realisation of any of their deliveries, \textit{i.e.,} the luxury
of telling a supplier $h$ what to do, conditional on the deliveries obtained
from other suppliers, is not available to her. Thus her possible strategies
are indexed by subsets $S\subseteq H$, where the $S$\emph{-strategy} consists
of making a pooling contract with each supplier in $S$, and a separating
contract with all the other suppliers.



\bigskip

\emph{Which strategy $S\subseteq H$ will maximize the entrepreneur's expected
output? }

\bigskip

There are clear-cut advantages of diversifying risk by means of separating
contracts. For suppose that many shipments fail under pooling, so that the
total $x$ and $y$ received by the entrepreneur are both small. Had she opted
for separating contracts, there would have remained a chance of less failure
of the $y$-shipments despite the widespread failure of the $x$-shipments,
enabling the production of medium output. Under pooling, $x$ and $y$ are
inexorably linked: if $x$ is small, so is $y,$ and therefore so is the output.
Why should the entrepreneur not diversify risk, instead of putting all the
eggs (inputs of $h$) in the same basket (shipment of $h$)? However we show
that she should do just that.

Indeed, consider the special case of a single supplier, with $H=\{1\}$. Then
the expected output is $px_{1}^{\alpha}y_{1}^{\beta}$, where $p$ is the
probability that the entrepreneur receives \emph{both} $x$ and $y$ from the
supplier. This probability is $p_{1}$ under pooling, and $p_{1}^{2}\leq p_{1}$
under separate shipments. Thus pooling leads to higher expected output.

The advantage of pooling persists with many suppliers. In fact, if $\Pi
_{1}\left(  S\right)  $ is the expected output under the $S$-strategy, then
one has the following stronger result.

\begin{proposition}
\label{Production Claim} The expected output $\Pi_{1}(S)$ is an increasing
function of $S$. In particular, the $H$-strategy maximizes the expected output.
\end{proposition}

\begin{proof}
See Section \ref{sec: Production Claim}.
\end{proof}




\section{Example 2: Military tactics}

Country I has two communication networks, $R$ (red) and $B$ (blue). Each
network is characterized by a pair $\left(  H_{\alpha},\mathcal{C}_{\alpha
}\right)  ,$ $\alpha\in\left\{  R,B\right\}  $. Here $H_{\alpha}$ is a finite
set of sites across the country at which there exist hubs of network $\alpha;$
and $\mathcal{C}_{\alpha}$ is a collection of \emph{critical subsets }of
$H_{\alpha}$, so called because the destruction of all the hubs of $\alpha$ in
$S\in\mathcal{C}_{\alpha}$ will disable network $\alpha$. It is natural to
assume, and we will, that the $\mathcal{C}_{\alpha}$ are \emph{increasing }in
the sense\emph{ }%
\begin{equation}
S\in\mathcal{C}_{\alpha}\text{ \& }S\subset T\implies T\in\mathcal{C}_{\alpha
}. \label{=monotonic collection}%
\end{equation}
Country II needs to disable \emph{both} networks. It knows only the set
$H=H_{R}\cup H_{B}$ of all sites and has no information regarding the pairs
$\left(  H_{\alpha},\mathcal{C}_{\alpha}\right)  $ other than that they exist.
For each $h\in H$ it has weapons, $r_{h}$ and $b_{h}$, which can destroy a red
hub or a blue hub respectively at site $h$, provided the hub exists
there.\footnote{e.g., the hubs of $R$ (or, $B$) are above (or, below) ground;
and $r_{h}$ detonates above, while $b_{h}$ burrows into the earth and
detonates below.} It also possesses several \textquotedblleft$h$%
-missiles\textquotedblright\ which are trained to fire $r_{h}$ and $b_{h}$ at
$h\in H;$ and each $h$-missile can carry either $r_{h}$ or $b_{h}$ or both,
but the probability that it will hit $h$ is $p_{h}$, independent of the
outcome of other missiles. Moreover, the missiles must be fired rapidly before
it can be known which ones hit their targets. The military is not concerned
with the costs of the weapons or of the missiles\footnote{However, see Remark
\ref{costs}.}. It is pondering over its $2^{\left\vert H\right\vert }%
$\ strategies, one for every $S\subset H,$ where the $S$-strategy consists of
firing the weapons jointly at the sites in $S$ and separately at those in
$H\diagdown S$.

\bigskip

\emph{Which strategy should the military employ? And would it be of benefit
for country II to conduct espionage to find out }$\left(  H_{\alpha
},\mathcal{C}_{\alpha}\right)  $\emph{ in order to fine-tune its strategy
based on that information?}

\bigskip

As in the previous example, there seem to be some advantages to firing
separately. For suppose the weapons are fired jointly at all $h$, and it turns
out that the set $S$ of targets hit is in $\mathcal{C}_{R}$ but not in
$\mathcal{C}_{B}$, with the upshot that Country II fails in its objective.
Perhaps its prospects might have been better had it fired separately.
Conditional on the realisation of $S$, there would still have remained a
positive probability of hitting a critical set $T$ in $\mathcal{C}_{B}$.
\medskip

Nevertheless, if $\Pi_{2}\left(  S\right)  $ denotes the probability that both
networks get disabled under the $S$-strategy then we prove the following.

\begin{proposition}
\label{Military Claim} The probability $\Pi_{2}\left(  S\right)  $ is an
increasing function of $S$. In particular, the $H$-strategy maximizes the
probability of disabling both networks.
\end{proposition}

\begin{proof}
See Section \ref{sec: Military Claim}.
\end{proof}

\begin{remark}
Note that this result holds no matter what $\left(  H_{R},\mathcal{C}%
_{R}\right)  $ and $\left(  H_{B},\mathcal{C}_{B}\right)  $ are. Thus
espionage is of no benefit. If the military \emph{were }concerned about the
costs of $r_{h},b_{h}$ then of course it would be of benefit to know
$H_{R},H_{B}$ in order to avoid firing $r_{h},b_{h}$ at $H\diagdown H_{R},$
$H\diagdown H_{B}$ respectively.
\end{remark}

\smallskip

It turns out that firing jointly increases not just the probability of
disabling both networks, but \emph{also }the probability of disabling neither.
What decreases is the probability of disabling exactly one network. Here is
the precise statement.

\begin{proposition}
\label{Military Claim 2} Let $F\left(  S\right)  $ and $G\left(  S\right)  $
be the probabilities of disabling neither network, or exactly one network,
under the $S$-strategy; then $F$ is an increasing function and $G$ is a
decreasing function of $S$.
\end{proposition}

\begin{proof}
See Section \ref{sec: Military Claim 2}.
\end{proof}

\section{Example 3: Corporate merger}
 Each individual $h\in H$ owns shares in company $A$ and/or $B$, and there are no shareholders outside of $H$. The \emph{voting game} (aka \textquotedblleft simple game\textquotedblright, see \cite{Shapley}%
)\ in each company is described by a function
\[
f_{\alpha}:\mathcal{H\longrightarrow}\left\{  0,1\right\}, \quad \alpha= A,B
\]
where $\mathcal{H}$ denotes the collection of all coalitions (subsets) of $H$.
The interpretation is that coalition $S\subset H$ is \emph{winning} (resp.,
\emph{losing}) \emph{in company }$\alpha$ if $f_{\alpha}\left(  S\right)  =1$
(resp., $f_{\alpha}\left(  S\right)  =0$). We naturally assume that
$f_{\alpha}$ is\emph{ increasing,} i.e., adding voters to a winning coalition
cannot make it losing; and (to avoid trivialities) that the empty coalition
$\emptyset$ is losing while the grand coalition $H$ is winning\footnote{A
canonical example is provided by a \emph{weighted voting game} in which player
$h$ has $r_{\alpha}^{h}\geq0$ votes in company $\alpha,$ equal to his shares
in $\alpha$ (adhering to the \textquotedblleft one dollar, one
vote\textquotedblright\ principle), and $f_{\alpha}\left(  S\right)  =1$ if,
and only if, $\sum_{\alpha\in S}r_{\alpha}^{h}\geq q_{\alpha}$ for some
\textquotedblleft quota\textquotedblright\ $0<q_{\alpha}\leq\sum_{\alpha\in
H}r_{\alpha}^{h}.$}.

The possibility has arisen of the merger of $A$ and $B$ but this requires the
approval of owners of both companies, i.e., merger must be voted for by a
winning coalition in both $A$ and $B.$

We postulate (in the spirit of \cite{Banzhaf} and \cite{Penrose}, see also
\cite{Felsenthal-Machover} for a detailed survey) that each $h$ has an
exogenously given probability $p_{h}$ of returning his ballot when called upon
to vote $Y$ (yes) in favor of merger; and that these probabilities are
independent across $h\in H$ and across the different occasions on which any
particular $h$ may be asked to return a ballot.

The management of both companies are strongly in favor of the merger. They can
send two separate ballots to any $h\in H,$ one of which represents a vote of $Y$ in
company $A,$ and the other a vote of $Y$ in $B$ (and this does seem to
be the norm in practice, where the decision-making inside any company is kept
independent of other companies). However, an interesting alternative presents
itself in the current context: they can send a \textquotedblleft joint
ballot\textquotedblright\ to $h,$ with the explicit understanding that if $h$
returns that ballot it will mean that $h$ voted $Y$ in \emph{both} $A$ and
$B.$

Thus, yet again, there is an $S$-strategy for every $S\subset H$: send joint
ballots to $h\in S$ and separate ballots to $h\in H\diagdown S$. If $\Pi
_{3}\left(  S\right)  $ is the probability of merger under the $S$-strategy
then the management is looking to maximize $\Pi_{3}\left(  S\right)  $.

\medskip

\emph{Which strategy $S\subset H$ will maximize the probability of merger?}

\bigskip

The advantages of risk-diversification notwithstanding, we have, as before:

\begin{proposition}
\label{Merger Claim} The merger probability $\Pi_{3}\left(  S\right)  $ is an
increasing function of $S$. In particular, the $H$-strategy maximizes the
probability of a merger.
\end{proposition}

\begin{proof}
See Section \ref{sec: Merger Claim}.
\end{proof}

\section{Example 4: \ Stochastic production - many inputs}

\label{sec:many}

Although the previous examples were binary in nature---two commodities, two
networks, two companies---this was merely for ease of exposition. In fact our
results hold in greater generality. To demonstrate this, in this section we
consider the case of stochastic production with many inputs---from a finite
commodity set $K$---and with a more general class of production functions that
includes the Cobb-Douglas functions as a special case.

\medskip An entrepreneur sources her inputs from a set $H$ of suppliers.
Supplier $h$ agrees to supply commodities from a subset $K^{h}$ of $K$ --- the
sets $K^{h}$ need not be disjoint. She further enters into a \textquotedblleft
shipping contract\textquotedblright\ $P^{h}$ with $h$, which is a partition of
$K^{h}$ into a disjoint union of shipments $K^{h}=K_{1}^{h}\sqcup\cdots\sqcup
K_{l}^{h}$. However there is uncertainity in shipping and shipment $K_{i}^{h}$
arrives with independent probability $p_{h}$.

After receiving the shipments, she produces an output given by%
\[
F\left(  \mathbf{S}\right)  =\prod\nolimits_{k\in K}F_{k}\left(  S_{k}\right)
,
\]
where $S_{k}$ is the set of players whose shipment of $k$ arrives
successfully, $\mathbf{S}$ is the \textquotedblleft success\textquotedblright%
\ tuple $\left(  S_{k}\right)  _{k\in K}$, and the $F_{k}$ are non-negative
increasing functions on subsets of $H$, that is to say $F_{k}\left(  S\right)
\geq F_{k}\left(  S^{\prime}\right)  $ if $S\supset S^{\prime}$. Her goal is
to choose the contract tuple $\mathbf{P}=\left(  P^{h}\right)  _{h\in H}$
which maximizes the expected output $\Phi\left(  \mathbf{P}\right)  $.

We now formulate a generalization of Proposition \ref{Production Claim}, which
involves a partial order on shipping contract tuples. If $\mathbf{P}=\left(
P^{h}\right)  _{h\in H}$ and $\mathbf{Q}=\left(  Q^{h}\right)  _{h\in H}$ are
two such tuples we write $\mathbf{P}\succeq\mathbf{Q}$ if $P^{h}$ is
\emph{coarser} than $Q^{h}$ for all $h$, that is if each $P^{h}$-shipment is a
union of $Q^{h}$-shipments. The maximal element for $\mathbf{\succeq}$ is the
\textquotedblleft coarse\textquotedblright\ tuple $\mathbf{C}=\left(
C^{h}\right)  _{h\in H}$ where $C^{h}$ is a single shipment of all commodities
in $K^{h}.$

\begin{proposition}
\label{production claim 2} The expected output is monotonic: $\mathbf{P}%
\succeq\mathbf{Q}$ implies $\Phi\left(  \mathbf{P}\right)  \geq\Phi\left(
\mathbf{Q}\right)  $. In particular, the coarse tuple $\mathbf{C}$ maximizes
expected output.
\end{proposition}

\begin{proof}
This is proved in Section \ref{sec:pc2}, where it is deduced from the more
general game-theoretic considerations of the next section.
\end{proof}

\smallskip

\begin{remark}
Suppose supplier $h$ agrees to send $x_{k}^{h}$ units of commodity $k$. Then
the total amount of $k$ received by the entrepreneur is $x_{k}=\sum_{h\in
S_{k}}x_{k}^{h}$. If we set $F_{k}\left(  S_{k}\right)  =\left(  x_{k}\right)
^{\alpha_{k}}$ for $a_{k}>0$ then the $F_{k}$ are increasing functions, and
$F\left(  \mathbf{S}\right)  =\prod_{k\in K}\left(  x_{k}\right)  ^{\alpha
_{k}}$ is the Cobb-Douglas production function.
\end{remark}

\section{A game-theoretic generalization \label{GTperspective}}

The ideas of this paper also have applications to game theory. To demonstrate
this we now recast the previous example as a strategic game among the suppliers.

\medskip We use the same setup, $K,H,K^{h},P^{h},p^{h},S_{k},\ldots$, with two
key differences.

\begin{enumerate}
\item We regard the partition $P^{h}$ of $K^{h}$ as a \emph{strategic} choice
of shipments by $h$.

\item If the success tuple is $\mathbf{S}=\left(  S_{k}\right)  _{k\in K}$
then player $h$ receives the payoff
\[
F^{h}\left(  \mathbf{S}\right)  =\prod\nolimits_{k\in K}F_{k}^{h}\left(
S_{k}\right)  ,
\]
where the $F_{k}^{h}$ are non-negative increasing functions that \emph{may}
depend on $h$.
\end{enumerate}

We will show that in this game\footnote{See Remarks \ref{NE and dominance},
\ref{ex post optimality}, \ref{incentive compatibility} for more details on
the game.} the coarse strategy $\mathbf{C}$ is \emph{dominant} in a very
strong sense. For this we consider the following scenario. Suppose the
shipping tuple $\mathbf{P}$ has been played, in which $P^{h}$ corresponds to
the partition $K^{h}=K_{1}^{h}\sqcup\cdots\sqcup K_{l}^{h}$ with $l\geq2$.
Player $h$ is informed of the success/failure of all shipments, including his
own, except for $K_{1}^{h}$ and $K_{2}^{h}.$ Let $\pi$ be his expected payoff
conditional on this information, and let $\pi^{\prime}$ be his expected payoff
if he chooses to combine the commodities in $K_{1}^{h}$ and $K_{2}^{h}$ into a
single shipment before sending them off.

\begin{proposition}
\label{Key Proposition} In the above scenario we have $\pi^{\prime}\geq\pi$.
\end{proposition}

\begin{proof}
See Section \ref{sec: Key Proposition}.
\end{proof}

As an immediate consequence we obtain the following result.

\begin{proposition}
\label{thm: dominant} Fix a choice of strategies by all players except $h$ and
let $\pi(P),\pi(Q)$ be the expected payoffs to $h$ if he chooses the
partitions $P,Q$, respectively.

If $P$ is coarser than $Q$ then we have $\pi(P)\ge\pi(Q)$. In particular, the
coarse partition is a dominant strategy for every player.
\end{proposition}

\begin{proof}
See Section \ref{sec: thm: dominant}.
\end{proof}

\section{Proofs}

\label{Proofs} We first establish a key mathematical result that underlies our
examples. This generalizes the well-known Harris inequality from extremal combinatorics.

Let $\mathcal{H}$ be the set of subsets of a finite set $H$. For $S$ in
$\mathcal{H}$ we define a probabilty measure $\mu_{S}$ on $\mathcal{H}%
\times\mathcal{H}$ as follows: $\mu_{S}(S_{1},S_{2})=0$ if $S\cap S_{1}\ne
S\cap S_{2}$, otherwise
\[
\mu_{S}(S_{1},S_{2})=P(S,S_{1})P(H\setminus S,S_{1}) P(H\setminus S,S_{2})
\]
where $P(X,Y):=\prod_{h\in X\cap Y}(p_{h})\prod_{h\in X\setminus Y}(1-p_{h})$.
Then $\mu_{S}(S_{1},S_{2})$ is precisely the probability that the following
random procedure leads to the pair $(S_{1},S_{2})$:

\begin{itemize}
\item For each $h$ in $S$ we toss a coin with probability $p_{h}$ of landing
heads; if heads we include $h$ in both $S_{1}$ and $S_{2}$, and if tails then
we exclude $h$ from both sets.

\item For each $h$ in $H\setminus S$ we toss the coin once to decide whether
to include $h$ in $S_{1}$, and once again, independently, to decide whether to
include $h$ in $S_{2}$.
\end{itemize}

\begin{definition}
If $f$ and $g$ are real valued functions on $\mathcal{H}$ then we define
\begin{equation}
f\star g(S):=\sum\nolimits_{S_{1},S_{2}\subseteq H}f(S_{1})g(S_{2})\mu
_{S}(S_{1},S_{2}). \label{=ES}%
\end{equation}

\end{definition}

As before, we say that $f$ is \emph{increasing }if $S\supseteq T$ implies
$f\left(  S\right)  \geq f\left(  T\right)  $.

\begin{theorem}
\label{th:main}If $f$ and $g$ are increasing functions then so is $f\star g.$
\end{theorem}

We first prove a special case of the result.

\begin{lemma}
\label{lem-key}Theorem \ref{th:main} holds for the case $|H|=1$.
\end{lemma}

\begin{proof}
Let $H$ be the singleton set $\{h\}$, let $p=p_{h}$ and write
\[
a = f\left(  \emptyset  \right)  , \; b = g\left(  \emptyset
\right)  ,\; c=f\star g \left(\emptyset  \right)  ,\quad a^{\prime
}=f\left(  H  \right)  ,\; b^{\prime}=g\left( H  \right)  ,\; c^{\prime}=f\star g \left(  H
\right)  .
\]
Then we need to show that $c^{\prime}-c \ge0$. However by an easy calculation
we have
\begin{gather*}
c=\left[  (1-p)a+pa^{\prime}\right]  \left[  (1-p)b+pb^{\prime}\right]  ,
\quad c^{\prime}=\left(  1-p\right)  ab+pa^{\prime}b^{\prime},\\
c^{\prime}-c=p\left(  1-p\right)  \left(  a^{\prime}-a\right)  \left(
b^{\prime}-b\right)  .
\end{gather*}
Since $f$ and $g$ are increasing we have $a^{\prime}\geq a,\; b^{\prime}\geq
b$, and thus $c^{\prime}-c\ge0.$
\end{proof}

\medskip

\begin{proof}
[Proof of Theorem \ref{th:main}]It clearly suffices to show that $f\star
g\left(  S\right)  \geq f\star g\left(  T\right)  $ in the case where
$S\setminus T$ consists of a single element, $h$, say. Then the two coin
tossing procedures agree on $H^{\prime}=H\setminus\{h\}$ and by grouping terms
we can rewrite \medskip%
\begin{align}
f\star g\left(  S\right)   &  =\sum\nolimits_{T_{1},T_{2}\subseteq H^{\prime}}
f_{T_{1}}\star g_{T_{2}}\left(  \{ h \} \right)  \mu_{T}^{\prime}(T_{1}%
,T_{2})\label{=A}\\
f\star g\left(  T\right)   &  =\sum\nolimits_{T_{1},T_{2}\subseteq H^{\prime}%
}f_{T_{1}}\star g_{T_{2}}\left(  \emptyset \right)  \mu_{T}^{\prime}(T_{1},T_{2}),
\label{=B}%
\end{align}
where $\mu_{T}^{\prime}$ is the measure on pairs of subsets of $H^{\prime}$
induced by $T$, $f_{T_{1}}$ and $g_{T_{2}}$ are function on subsets of the
singleton set $\{h\}$ defined by
\[
f_{T_{1}}\left( \emptyset\right)  =f\left(  T_{1}\right)  ,\; f_{T_{1}}\left(
\{ h \} \right)  =f\left(  T_{1} \cup\left\{  h\right\}  \right)  ,\;
g_{T_{2}}\left(  \emptyset \right)  =g\left(  T_{2}\right)  ,\; g_{T_{2}}\left(
\{ h \} \right)  =g\left(  T_{2} \cup\left\{  h\right\}  \right)  ,
\]
and the convolution structure on $\{h\}$ corresponds to $p=p_{h}$. By Lemma
\ref{lem-key} each term in \eqref{=A} dominates the corresponding term in
\eqref{=B}, and thus the result follows
\end{proof}

\bigskip

For the two extreme cases $S=H$ and $S=\emptyset$ we have%
\[
f\star g\left(  H\right)  =\operatorname*{Exp}(fg),\quad f\star g(\emptyset
)=\operatorname*{Exp}(f)\operatorname*{Exp}(g),
\]
where $\operatorname*{Exp}$ is the expectation with respect to the measure
$\mu$ on $\mathcal{H}$ defined by
\begin{equation}
\mu(S)=%
{\textstyle\prod\nolimits_{h\in S}}
p_{h}%
{\textstyle\prod\nolimits_{h\notin S}}
(1-p_{h})=P(H,S). \label{=prodmeas}%
\end{equation}
Thus we obtain the following well-known inequality due to Harris \cite{Harris}.

\begin{corollary}
\label{AH} If $f$ and $g$ are increasing then we have
\begin{equation}
\operatorname*{Exp}(fg)\geq\operatorname*{Exp}(f)\operatorname*{Exp}(g).
\label{generalized Harris}%
\end{equation}

\end{corollary}

If $\mathcal{F}=\{f_{i},i\in I\}$ is an \emph{set }of functions and $\pi
:I_{1}\sqcup\cdots\sqcup I_{k}$ is a partition of the index set $I$ then we
define $E_{\mathcal{F}}\left(  \pi\right)  =\operatorname*{Exp}\left(
\prod\nolimits_{i\in I_{1}}f_{i}\right)  \cdots\operatorname*{Exp}\left(
{\prod\nolimits_{i\in I_{k}}}f_{i}\right)  .$

\begin{corollary}
\label{Corollary} Let $\mathcal{F}$ be a finite set of non-negative increasing
functions and let $\pi$ and $\pi^{\prime}$ be partitions of $I$ such that
$\pi^{\prime}$ refines $\pi$ then we have $E_{\mathcal{F}}\left(  \pi^{\prime
}\right)  \leq E_{\mathcal{F}}\left(  \pi\right)  .$
\end{corollary}

\begin{proof}
If $f_{1},\ldots,f_{k}$ are nonnegative and increasing, then so is their
product. Now the result follows by iterated application of
(\ref{generalized Harris}).
\end{proof}
\\

We now give proofs of all the propositions in the previous sections.

\subsection{Proof of Proposition \ref{Production Claim}}

\label{sec: Production Claim}

\begin{proof}
It is easy to see that the output for the $S$-strategy is $\Pi_{1}=F_{1}\star
F_{2}$. Thus the result follows from Theorem \ref{th:main}.
\end{proof}


\subsection{Proof of Proposition \ref{Military Claim}}

\label{sec: Military Claim}

\begin{proof}
Let $f$ and $g$ be the characteristic functions of $\mathcal{C}_{R}$ and
$\mathcal{C}_{B}$, then by assumption $f,g$ are increasing. We have
$f(S_{1})g(S_{2})=1$ or $0$ according as the pair $(S_{1},S_{2})$ does or does
not succeed in disrupting both networks. It follows that the success
probability for the $S$-strategy is $\Pi_{2}=f\star g$, which is increasing by
Theorem \ref{th:main}.
\end{proof}

\subsection{Proof of Proposition \ref{Military Claim 2}}

\label{sec: Military Claim 2}

\begin{proof}
The probabilty $F(S)$ of disabling neither network is given by
\[
F=(1-f)\star(1-g) =(f-1)\star(g-1),
\]
where $f$ and $g$ are as in the previous proof. Since $(f-1)$ and $(g-1)$ are
increasing functions so is $F$. This implies that $G =1-\Pi_{2}-F$ is a
decreasing function.
\end{proof}

\subsection{Proof of Proposition \ref{Merger Claim}}

\label{sec: Merger Claim}

\begin{proof}
It is easy to see that the success probability of the $S$-strategy is $\Pi
_{3}=f_{A}\star f_{B}$, where $f_{A}$ and $f_{B}$ are the voting functions of
companies $A$ and $B$. Now $f_{A}$ and $f_{B}$ are increasing by assumption
and thus by Theorem \ref{th:main} so is $\Pi_{2}$.
\end{proof}

\subsection{Proof of Proposition \ref{Key Proposition}}

\label{sec: Key Proposition}

\begin{proof}
The proof involves the same idea as Lemma \ref{lem-key}. The details are as follows.

Let $K^{\prime}$ be the set of commodities not in $K_{1}^{h}\cup K_{2}^{h}$,
then the sets $S_{k},k\in K^{\prime}$ are constant under the conditioning
assumption, and so is the product $c=\prod_{k\in K^{\prime}}F_{k}^{h}\left(
S_{k}\right)  $. The products $a=\prod_{k\in K_{1}^{h}}F_{k}^{h}\left(
S_{k}\right)  $ and $b=\prod_{k\in K_{2}^{h}}F_{k}^{h}\left(  S_{k}\right)  $
have two possible values $a_{1}\geq a_{0}$ and $b_{1}\geq b_{0}$ corresponding
to the arrival/non-arrival of shipments $K_{1}^{h}$ and $K_{2}^{h}$,
respectively. When sent separately these arrive with (independent) probability
$p=p_{h},$ while sent together they arrive together with probability $p$. The
payoff is the expectation of the product $abc$, and so we get%
\[
\pi^{\prime}=\left[  \left(  1-p\right)  a_{0}b_{0}+pa_{1}b_{1}\right]
c,\text{\quad}\pi=\left[  (1-p)a_{0}+pa_{1}\right]  \left[  (1-p)b_{0}%
+pb_{1}\right]  c.
\]
Now a straightforward algebraic calculation shows that%
\[
\pi^{\prime}-\pi=p\left(  1-p\right)  \left(  a_{1}-a_{0}\right)  \left(
b_{1}-b_{0}\right)  c
\]
Every factor in this expression is non-negative, which implies $\pi^{\prime
}-\pi\geq0$.
\end{proof}

\subsection{Proof of Proposition \ref{thm: dominant}}

\label{sec: thm: dominant}

\begin{proof}
We can go from any partition to a coarser partition in a sequence of steps,
where at each step we combine two parts of a partition into a single part.
Clearly it is enough to show that the desired inequality holds at each step of
this procedure. But this follows from Proposition \ref{Key Proposition}.
\end{proof}

\subsection{Proof of Proposition \ref{production claim 2}}

\label{sec: production claim 2}

\label{sec:pc2}

\begin{proof}
We specialize the strategic model to the "symmetric" case in which $F_{k}%
^{h}=F_{k}$ for all $h$. Then each player's payoff is the same as that of the
principal in stochastic production model of the previous section. Thus the two
optimization problems are the same, and hence Proposition
\ref{production claim 2} follows from Proposition \ref{thm: dominant}.
\end{proof}

\section{Remarks}

\subsection{Robust Optimality\label{robust optimality}}

Pooling is optimal \emph{for all possible characteristics }of the population
under consideration, e.g., $\left(  x_{h},y_{h},p_{h}\right)  _{h\in H}$ in
Example 1, $\left(  H_{h},\mathcal{C}_{h}\right)  _{h\in\left\{  A,B\right\}
}$ in Example 2, and so on. In short, the $H$-strategy is not only optimal but
\emph{robustly optimal}. And therein lies its full value: pooling can be
invoked with impunity, without having detailed information of the population characteristics.

\subsection{Costs\label{costs}}

We have assumed various monetary costs (such as those of firing missiles in
Example 1, or of shipping inputs in Example 2, etc.) to be zero. Suppose these
costs existed, but with \textquotedblleft economies of scale\textquotedblright%
, i.e., the cost of joint action is less than the sum of the costs of the
separate actions. Then the pooling strategy is even more efficient, as it
diminishes costs, \emph{over and above} the favorable probabilistic
implications of Theorem \ref{th:main}.

\subsection{\textbf{Nash Equilibrium and Dominant
Strategies\label{NE and dominance}}}

For completeness' sake, we review some standard game-theoretic definitions.
First recall that once each supplier $h$ chooses his strategy $P^{h}$ (i.e., a
partition of $K^{h}),$ a probability distribution is induced on success tuples
$\mathbf{S=}\left(  S_{k}\right)  _{k\in K}\in\mathcal{H}^{K}$ via the
independent arrivals of the elements of the partitions in $\mathbf{P}=\left(
P^{h}\right)  _{h\in H}$; and that the payoff to $h$ is the expectation
$\Phi^{h}\left(  \mathbf{P}\right)  $ of $\prod\nolimits_{k\in K}F_{k}%
^{h}\left(  S_{k}\right)  $ w.r.t. this probability distribution. We are now
ready for the definitions (and the subsequent remarks below). A strategy
$P^{h}$ of player $h$ is a \emph{best reply} to the strategy-selection
$\left(  P^{i}\right)  _{i\in H\diagdown\left\{  h\right\}  }$ of the other
players, if $P^{h}$ maximizes the payoff to $h$ over all his strategies,
conditional on the fixed $\left(  P^{i}\right)  _{i\in H\diagdown\left\{
h\right\}  }.$ The $H$-tuple of strategies $\mathbf{P}=\left(  P^{h}\right)
_{h\in H}$ is a \emph{Nash Equilibrium }(NE) if $P^{h}$ is a best reply to
$\left(  P^{i}\right)  _{i\in H\diagdown\left\{  h\right\}  }$ for every $h\in
H.$ Next, $P^{h}$ is a \emph{dominant strategy }of $h$ if $P^{h}$ is a best
reply to \emph{every} strategy-selection of the players in $H\diagdown\left\{
h\right\}  $. Finally $P^{h}$ is a \emph{strictly dominant strategy }of $h$ if
$P^{h}$ is the \emph{unique} best reply to every strategy-selection by the
players in $H\diagdown\left\{  h\right\}  $. It is obvious that if $P^{h}$ is
a strictly dominant strategy for every $h\in H$, then $\mathbf{P}=\left(
P^{h}\right)  _{h\in H}$ is the unique NE of the game. It is also obvious that
if each $F_{k}^{h}$ is assumed to be \emph{strictly }increasing, then the
coarse partition is \emph{strictly }dominant for every player (This follows
from the fact that $\pi^{\prime}-\pi$ is strictly positive in the proof of
Proposition \ref{Key Proposition}). In this case the coarse partitions
constitute the \emph{unique} Nash Equilibrium of the game.

\subsection{\textbf{\textit{Ex Post} Optimality\label{ex post optimality}}}

Our proof of Proposition \ref{thm: dominant} shows that the coarse partition
not only maximizes the expectation of $\prod\nolimits_{k\in K}F_{k}^{h}\left(
S_{k}\right)  $ for player $h$, but in fact maximizes his expected payoff
conditional on \emph{every realization} of the success tuples in
$\widetilde{\mathcal{H}}^{K}$ due to the other players (where
$\widetilde{\mathcal{H}}$ denotes the collection of subsets of $H\diagdown
\left\{  h\right\}  $). In other words, the coarse strategy is not just
dominant in the standard sense (\textit{ex ante} optimal), but in a much
stronger sense (\textit{ex post} optimal).

\subsection{\textbf{Incentive Compatibility\label{incentive compatibility}}}

Proposition \ref{thm: dominant} considerably bolsters the economic
plausibility of our examples. To this end, consider Example 4. It is not easy
for the entrepreneur to \emph{implement }the optimal contract with her agents.
She would need to monitor their behavior and institute costly punishment for
anyone who breaks the contract, i.e., takes it into his head to choose a
partition other than the coarsest. Moreover she would need to make it credible
to her agents that the punishment will be forthcoming so that it acts as a
deterrent. All this is not achieveable without considerable effort and
cost to herself, if it is achievable at all. However, our game-theoretic
approach provides a way out of this impasse. The entrepreneur can simply
announce that she will part with a (tiny!) fraction $\kappa^{h}$ of her profit
$\prod\nolimits_{k\in K}F_{k}\left(  S_{k}\right)  $ to each supplier $h.$
This engenders a game among the suppliers with payoffs $F^{h}\left(
\mathbf{S}\right)  =\kappa^{h}\prod\nolimits_{k\in K}F_{k}\left(
S_{k}\right)  $ to $h\in H$, {\it i.e.,} a game with \emph{common }payoffs (up to scalar
multiplication). In this game it is a dominant strategy for each $h$ to choose
his coarsest partition by Proposition \ref{thm: dominant}. Thus the optimal
contract of the principal is implemented by her agents of their own accord, at
virtually no cost to her!


\begin{thebibliography}{99}                                                                                               %


\bibitem {AD}Ahlswede, R. \& D.E. Daykin (1978). An inequality for the weights
of two families of sets, their unions and intersections.
\textit{Z.Wahrscheinlichkeitstheorie ind Verw. Gebeite 43,183-185.}

\bibitem {Alon-Spencer}Alon, N. \& J.Spencer (2016). The Probabilistic Method,
\textit{Fourth Edition, John Wiley \& Sons, Inc., Hoboken, New Jersey.}

\bibitem {Athey}Athey, S. (2002). Monotone Comparative Statics Under
Uncertainty, \textit{Quarterly Journal of Economics 117 (1), 187-223.}

\bibitem {Azar}Azar, Y., N. Devanur, K. Jain, and Y. Rabani (2010).
Monotonicity in bargaining networks. \textit{In Proceedings of the
twenty-first annual ACM-SIAM symposium on Discrete algorithms (SODA '10).
Society for Industrial and Applied Mathematics, USA, 817-826.}

\bibitem {Banzhaf}Banzhaf, J.F. (1965). Weighted Voting Doesn't Work: A
Mathematical Analysis, \textit{Rutgers Law Review 19 (2), 317-343.}

\bibitem {Felsenthal-Machover}Felsenthal, D.S. \& M. Machover (2004). A Priori
Voting Power: What is it all About? \textit{Political Studies Review 2 (1),
1-23 }

\bibitem {Fortuin}Fortuin, C.M., R.W. Kasteleyn \& J. Ginibre (1971).
Correlation Inequalities on Some Partially Ordered Sets.
\textit{Communications in Mathematical Physics, 22 (2), 89-103.}

\bibitem {Harris}Harris, T.E. (1960). A Lower Bound for the Critical
Probability in a Certain Percolation Process. \textit{Mathematical Proceedings
of the Cambridge Philosophical Society 56(1), 13-20.}

\bibitem {Nikzad}Nikzad, A. (2022). Rank-optimal assignments in uniform
markets \textit{Theoretical Economics 17 , 25-55 }

\bibitem {Penrose}Penrose, L. (1946). The Elementary Statistics of Majority
Voting, \textit{Journal of the Royal Statistical Society 109 (1), 53-57.}

\bibitem {Postlewaite}Postlewaite, A., L. Samuelson \& D. Silverman (2008).
Consumption Commitments and Employment Contracts. \textit{Review of Economic
Studies 75, 559-578}

\bibitem {Quah}Quah, J. (2012). Aggregating the Single-Crossing Property,
\textit{Econometrica, 80 (5), 2333-2348.}

\bibitem {Shapley}Shapley, L.S. (1954). Simple Games: An Outline of the
Descriptive Theory, \textit{Rand Corporation Research Memorandum, RM-1384.}
\end{thebibliography}
\end{document}